\documentclass{article}

\usepackage{amsmath,amssymb}
\usepackage{graphicx}
\usepackage{wrapfig}
\usepackage{here}
\usepackage{amsthm}
\usepackage{cancel}
\usepackage{mathtools}
\usepackage[all]{xy}
\usepackage{diagbox}
\usepackage{physics}
\usepackage[all]{xy}

\newtheorem{dfn}{Definition}
\newtheorem{lem}{Lemma}
\newtheorem{prop}{Proposition}
\newtheorem{thm}{Theorem}
\newtheorem{cor}{Corollary}
\newtheorem{exm}{Example}
\newtheorem{rmk}{Remark}

\usepackage{color}

\ifx \ocirc \undefined \def \ocirc #1{{\accent'27#1}}\fi

\title{Sufficiency of Unit Coefficients for Binary Orbits in Uniformly Weighted Linear Cellular Automata}
\author{Akane Kawaharada\thanks{
Address: 1-10-20, Hashirimizu, Yokosuka, Kanagawa, 239-8686, JAPAN.  
E-mail: aka@nda.ac.jp}
\\ \vspace{2pt}
Department of Mathematics, National Defense Academy of Japan}
\date{\today}

\begin{document}
\maketitle

\begin{abstract}
This paper investigates the classification of spatio-temporal patterns generated by linear cellular automata with uniform weights (LCA-UW) over the ring ${\mathbb Z} / n {\mathbb Z}$. 
While these systems are governed by the state size $n$ and a transition coefficient $c$, their combined influence produces a vast array of patterns that are difficult to organize through exhaustive observation.
We introduce a binary projection operator $\mathcal{B}$ to focus on the fundamental structural evolution (infinite binary orbits) of these automata. 
Our main result demonstrates a fundamental reduction principle.
For any coefficient $c$ that shares prime factors with $n$, the generated infinite binary orbit eventually coincides with the orbit of an LCA-UW with some reduced state size and a unit coefficient $c=1$. 
We prove that for a fixed $n$, there exist exactly $2^m - 1$ distinct types of binary orbits, where $m$ is the number of distinct prime factors of $n$. 
This theorem effectively collapses the two-dimensional parameter space $(n, c)$ into a one-dimensional search over $n$, providing a streamlined framework for the topological and fractal classification of LCA-UW dynamics.
\end{abstract}

\hspace{2.5mm} 

\section{Introduction}
\label{sec:intro}

A cellular automaton (CA) is a simple mathematical model capable of generating complex patterns. 
Since these patterns often exhibit fractal properties, many researchers have used CAs as fractal generators \cite{willson1984, culik1989, takahashi1992, haeseler1993, wolfram2002}.
Typically, fractals are characterized by a scalar value known as a fractal dimension (e.g., the Hausdorff, Box-counting, or packing dimension). 
However, distinct fractals can sometimes share the exact same dimensional value, as a single scalar cannot capture the full complexity of their structures.
To address this limitation, we developed a novel method for classifying fractals by employing one-variable functions that describe the spatio-temporal dynamics of the CA patterns.  
This approach allows us to preserve much more structural information regarding the fractal's structure \cite{kawa2022P, kawa2024, kawa2025}.
 
Our research has primarily focused on linear cellular automata (LCAs), which generate complex, self-organizing patterns from local algebraic rules. 
Martin et al. \cite{martin1984} shown that the complexity of the resulting dynamics increases significantly when the state space is the ring ${\mathbb Z} / n {\mathbb Z}$.
While the rule space of LCAs is vast, previous studies have shown that it can be partitioned into equivalence classes through algebraic transformations \cite{cattaneo1997}. 
Building upon these ideas of reduction, a previous study \cite{kawa2024P} categorized LCA orbits by focusing on the ``seed'' of the initial configuration. 
It demonstrated that any orbit from a single-site initial configuration is either isomorphic to the orbit for seed $1$ of the same state number or to a seed $1$ orbit of a smaller state number. 
Thus, for the single-site initial configuration, it is sufficient to consider only the case of seed $1$.

In the present paper, we extend this simplification principle by focusing not only on initial configurations but also on transition rules.
We classify orbits by identifying cases where different transition rules produce the same binary behavior.
To make our analysis more efficient, we employ a binary projection operator $\mathcal{B}$.
This approach is more practical than traditional isomorphism-based classification, which often fails to capture similar patterns if their state labels do not match perfectly.
As illustrated in Figure~\ref{fig:0000}, two spatio-temporal patterns may not be isomorphic in their multi-state form, yet they can produce identical binary orbits when projected by $\mathcal{B}$. 
A classification based strictly on isomorphism is often too granular, preventing us from recognizing the common structure behind systems that seem different.
By mapping the $n$ possible states onto a logical field (typically $\{0, 1\}$), the operator $\mathcal{B}$ allows us to abstract away specific numerical values and focus on the asymptotic structural framework of pattern development.

\begin{figure}[htb]
\centering
\begin{minipage}[b]{0.49\columnwidth}
\centering
\includegraphics[width=0.9\columnwidth]{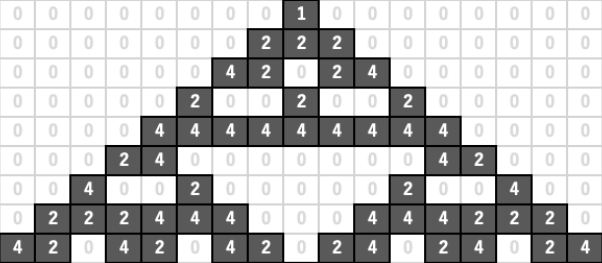}\\
(a)  $n=6$, $c=2$
\end{minipage}
\begin{minipage}[b]{0.49\columnwidth}
\centering
\includegraphics[width=0.9\columnwidth]{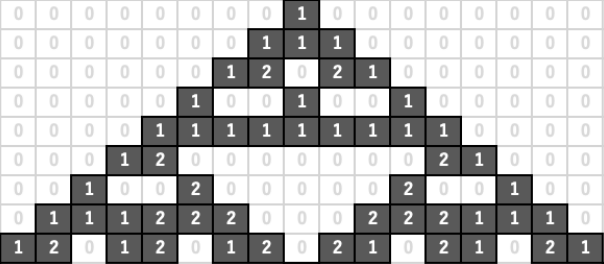}\\
(b) $n=3$, $c=1$
\end{minipage}
\caption{An example of non-isomorphic patterns that yield identical binary orbits.}
\label{fig:0000}
\end{figure}

Furthermore, we classify binary orbits based on their eventual coincidence. 
Our analysis reveals numerous cases where orbits differ during the first few steps but coincide perfectly thereafter. 
A representative example of this behavior is illustrated in Figure~\ref{fig:0001}. 
Although these two patterns are initially distinct, they eventually converge to the same binary orbit, suggesting that they belong to the same dynamical class.

\begin{figure}[htb]
\centering
\begin{minipage}[b]{0.49\columnwidth}
\centering
\includegraphics[width=0.9\columnwidth]{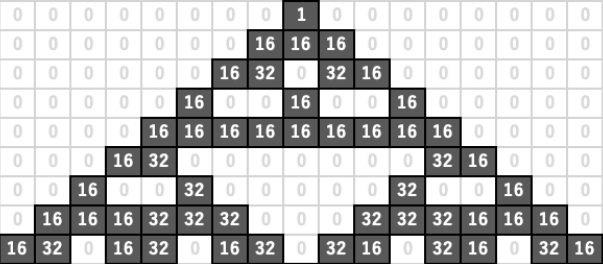}\\
(a) $n=48$, $c=16$
\end{minipage}
\begin{minipage}[b]{0.49\columnwidth}
\centering
\includegraphics[width=0.9\columnwidth]{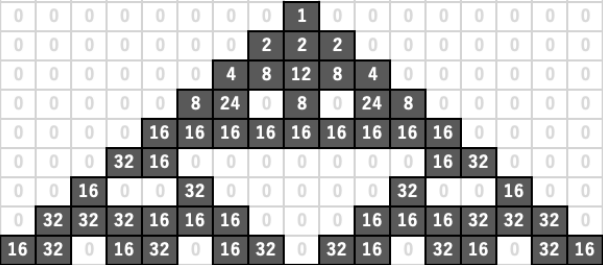}\\
(b) $n=48$, $c=2$
\end{minipage}
\caption{An example where orbits differ in the first few steps but eventually coincide in their binary representation.}
\label{fig:0001}
\end{figure}

This approach is particularly crucial in the study of linear cellular automata with uniform weights (LCA-UW), where the next state is determined by multiplying the neighborhood sum by a coefficient $c$. 
This introduces a two-dimensional parameter space $(n, c)$. 
Intuitively, one might expect that varying $c$ would generate an entirely new class of patterns for each $n$. 
However, our preliminary observations suggest that many regions within this parameter space are actually equivalent.
The purpose of this paper is to formally establish a reduction principle for binary orbits in LCA-UW based on their eventual coincidence. 
We show that for any number of states $n$, the infinite binary orbits generated by varying $c$ can be classified into a remarkably few number of finite types.
Crucially, all such orbits are eventually equivalent to a $c=1$ orbit for some effective state number. 
By proving this equivalence, we clarify that the entire binary patterns in LCA-UW can be fully understood by analyzing only the $c=1$ cases across all $n$. 
This reduction greatly simplifies the classification problem and provides a new perspective on the algebraic symmetries inherent in the dynamics of LCAs.

The remainder of this paper is organized as follows. 
Section~\ref{sec:pre} provides preliminaries on LCAs and reviews relevant established facts. 
Section~\ref{sec:main} presents our main results concerning a fundamental reduction principle.
We demonstrate that for any coefficient $c$ of LCA-UWs that shares prime factors with the number of states $n$, the generated infinite binary orbit eventually coincides with the orbit of an LCA-UW with a unit coefficient $c=1$ and a state size no greater than $n$.
Specifically, we show that the entire binary patterns of LCA-UWs can be fully characterized by examining only the cases where $c=1$ across all $n \in {\mathbb Z}_{\geq 2}$. 
Finally, Section~\ref{sec:cr} discusses our findings and highlights some possible avenues for future research.

\section{Preliminaries}
\label{sec:pre}

Throughout this paper, we adopt the following notation.

\begin{itemize}
\item For $n \in {\mathbb Z}_{\geq 2}$, let ${\mathbb Z} / n {\mathbb Z} := \{0, 1, \ldots, n-1\} \pmod n$ denote the ring of integers modulo $n$.
The order of this set is $|{\mathbb Z} / n {\mathbb Z}| = n$.
\item Let $({\mathbb Z} / n {\mathbb Z})^{\ast} := ({\mathbb Z} / n {\mathbb Z}) \setminus \{0\}$.
\item For $c \in ({\mathbb Z} / n {\mathbb Z})^{\ast}$, we define $c {\mathbb Z} / n {\mathbb Z} := \{ ca \pmod n \mid a \in {\mathbb Z} / n {\mathbb Z} \}$. 
This subset $c {\mathbb Z} / n {\mathbb Z}$ is the principal ideal generated by $c$ in the ring ${\mathbb Z} / n {\mathbb Z}$, and its order is given by $n/\gcd(c, n)$.
\item The unit group of ${\mathbb Z} / n {\mathbb Z}$ is denoted by $({\mathbb Z} / n {\mathbb Z})^{\times} := \{ a \in \mathbb{Z}/n\mathbb{Z} \mid \gcd(a, n) = 1 \}$.
\item Let $n = \prod_{j=1}^m p_j^{k_j}$ be the prime factorization of $n$, where each $p_j$ is a distinct prime and $k_j \in \mathbb{Z}_{>0}$ for $1 \leq j \leq m$. We define the maximum exponent $K_n := \max_{1 \leq j \leq m} \{ k_j \}$.
\end{itemize}

In this study, we consider the following CAs.

\begin{dfn}[Linear cellular automaton]
\label{dfn:LCA}
Let $J \in {\mathbb Z}_{\geq 1}$, $c_j \in ({\mathbb Z} / n {\mathbb Z})^{\ast}$ and $\vb*{r}_j \in {{\mathbb Z}^D}$ for $1 \leq j \leq J$, where $D \in {\mathbb Z}_{\geq 1}$.
For a configuration $x \in {({\mathbb Z} / n {\mathbb Z})}^{{\mathbb Z}^D}$, a $D$-dimensional \textbf{linear cellular automaton (LCA)} with $n$ states and coefficients $c_j \in ({\mathbb Z} / n {\mathbb Z})^{\ast}$, denoted by $T_{n, \{c_j\}} : {({\mathbb Z} / n {\mathbb Z})}^{{\mathbb Z}^D} \to {({\mathbb Z} / n {\mathbb Z})}^{{\mathbb Z}^D}$, is defined at each site $\vb*{i} \in {{\mathbb Z}^D}$ by
\begin{align}
(T_{n, \{c_j\}} x)_{\vb*{i}} = \sum_{j=1}^J c_j x_{\vb*{i} + \vb*{r}_j} \pmod n.
\end{align}

\end{dfn}
\begin{dfn}[LCA with uniform weights]
\label{dfn:LCA-UW}
Let $J \in {\mathbb Z}_{\geq 1}$, $c \in ({\mathbb Z} / n {\mathbb Z})^{\ast}$ and $\vb*{r}_j \in {{\mathbb Z}^D}$ for $1 \leq j \leq J$, where $D \in {\mathbb Z}_{\geq 1}$.
For a configuration $x \in {({\mathbb Z} / n {\mathbb Z})}^{{\mathbb Z}^D}$, a $D$-dimensional \textbf{linear cellular automaton with uniform weights (LCA-UW)} with $n$ states and a coefficient $c \in ({\mathbb Z} / n {\mathbb Z})^{\ast}$, denoted by $T_{n, c} : {({\mathbb Z} / n {\mathbb Z})}^{{\mathbb Z}^D} \to {({\mathbb Z} / n {\mathbb Z})}^{{\mathbb Z}^D}$, is defined at each site $\vb*{i} \in {{\mathbb Z}^D}$ by
\begin{align}
(T_{n, c} x)_{\vb*{i}} = \sum_{j=1}^J c x_{\vb*{i} + \vb*{r}_j} \pmod n.
\end{align}
\end{dfn}

In this paper, we focus exclusively on the case of LCA-UW, rather than general LCAs.
We consider the orbit starting from the following initial configuration $x_o \in {({\mathbb Z} / n {\mathbb Z})}^{{\mathbb Z}^D}$,
\begin{align}
(x_o)_{\vb*{i}} = 
\begin{cases} 
1 & \quad ({\vb*{i}} = {\vb*{0}}),\\
0 & \quad ({\vb*{i}} \neq {\vb*{0}}).
\end{cases}
\end{align}
We refer to this configuration $x_o$ as the \textbf{single-site seed (SSS)}.
For a given number of states $n$ and a coefficient $c$, we define the evolution of the LCA-UW as a discrete-time dynamical system. 
The orbit originating from the initial configuration SSS $x_o$ is denoted by $\{T_{n, c}^t x_o\}_{t \geq 0}$, where $T_{n, c}^t$ represents the $t$-th iterate of $T_{n, c}$.

According to the algebraic consideration by Martin et al. \cite{martin1984}, the state of an LCA generated from the initial configuration $x_o$ can be expressed in terms of multinomial coefficients as follows.

\begin{lem}
\label{lem:coef00}
The state at time step $t \in {\mathbb Z}_{\geq 0}$ and site $\vb*{i} \in {\mathbb Z}^D$ of the LCA-UW $T_{n, c}$ starting from the SSS $x_o$ is given by
\begin{align}
\label{eq:coef00}
(T_{n, c}^t x_o)_{\vb*{i}} &= 
\sum_{\substack{\sum_{j=1}^J t_j = t\\ \sum_{j=1}^J t_j \vb*{r}_j = - \vb*{i} }} \frac{t!}{t_1! t_2! \ldots t_J!} c^t \pmod n,
\end{align}
where the sum is taken over all sequences of non-negative integers $(t_1, t_2, \ldots, t_J)$ satisfying the given conditions.
\end{lem}

\begin{proof}
For $D \in {\mathbb Z}_{\geq 1}$ and $\vb*{r}_j := (r_{j 1}, r_{j 2}, \ldots, r_{j D}) \in {{\mathbb Z}^D}$ for $1 \leq j \leq J$, applying the multinomial theorem to $X_1, X_2, \ldots, X_D$, we obtain
\begin{align*}
& (c X_1^{r_{1 1}} X_2^{r_{1 2}} \cdots X_D^{r_{1 D}} + c X_1^{r_{2 1}} X_2^{r_{2 2}} \cdots X_D^{r_{2 D}} + \cdots + c X_1^{r_{J 1}} X_2^{r_{J 2}} \cdots X_D^{r_{J D}} )^t\\
&= c^t \left( \sum_{j=1}^J \prod_{d=1}^D X_d^{r_{j d}} \right)^t\\
&= c^t \sum_{\sum_{j=1}^J t_j = t} \frac{t!}{t_1! t_2! \ldots t_J!} \left( \prod_{d=1}^D X_d^{r_{1 d}} \right)^{t_1} \left( \prod_{d=1}^D X_d^{r_{2 d}} \right)^{t_2} \cdots \left( \prod_{d=1}^D X_d^{r_{D d}} \right)^{t_D}\\
&= \sum_{\sum_{j=1}^J t_j = t} \frac{t!}{t_1! t_2! \ldots t_J!} c^t \left( X_1^{\sum_{j=1}^J t_j r_{j 1}} X_2^{\sum_{j=1}^J t_j r_{j 2}} \cdots X_D^{\sum_{j=1}^J t_j r_{j D}} \right).
\end{align*}
The state at site $\vb*{i} := (i_1, i_2, \ldots, i_D) \in \mathbb{Z}^D$ at time step $t$ corresponds to the coefficient of the term $\prod_{d=1}^D X_d^{-i_d}$ in the above expansion (noting the definition of the LCA uses the state at $\vb*{i} + \vb*{r}_j$). Therefore, summing over all sets of non-negative integers $(t_1, t_2, \dots, t_J)$ satisfying $\sum_{j=1}^J t_j = t$ and $\sum_{j=1}^J t_j \vb*{r}_j = -\vb*{i}$ yields Equation~\eqref{eq:coef00} modulo $n$.
\end{proof}

In Examples~\ref{exm:coef01} and \ref{exm:coef02}, we demonstrate that Lemma~\ref{lem:coef00} holds for specific instances.

\begin{exm}
\label{exm:coef01}
As an application of Lemma~\ref{lem:coef00}, we consider the case where $D=1$, $J=3$, with neighborhood offsets $r_1=-1$, $r_2=1$, and $r_3=2$.
Here, we consider an LCA-UW defined by $(T_{n, c} x)_i = c (x_{i-1} + x_{i+1} + x_{i+2}) \pmod n$. 
In this setting, the condition $\sum_{j=1}^3 t_j r_j = -i$ in Equation~\eqref{eq:coef00} becomes $-t_1 + t_2 + 2t_3 = -i$, or equivalently, $t_1 - t_2 - 2t_3 = i$. 
Starting from the SSS $x_o$, the state at time $t$ is
\begin{align}
\label{eq:coef01}
(T_{n, c}^t x_o)_i &= 
\sum_{\substack{t_1 + t_2 + t_3 = t\\t_1 - t_2 - 2 t_3 = i}} \frac{t!}{{t_1}! {t_2}! {t_3}!} c^t \pmod n.
\end{align}
Calculating the values of each cell up to $t=3$ based on Equation~\eqref{eq:coef00} yields Table~\ref{tab:coef01}. 
It can be confirmed that the time evolution shown in Table~\ref{tab:st_m01} precisely follows the orbit of the LCA-UW.
\begin{table}[H]
\caption{Values of each site $i$ for the orbit of the LCA-UW $(T_{n, c} x)_i = c (x_{i-1} + x_{i+1} + x_{i+2}) \pmod n$ from the SSS $x_o$ at time steps $t=0, 1, 2, 3$.}
\label{tab:coef01}
\centering
\vspace{1mm}
\begin{tabular}{c | c | c | c | l}
$t$ & $(t_1, t_2, t_3)$ & $i$ & $\frac{t!}{{t_1}! {t_2}! {t_3}!} c^t$ & $(T_{n, c}^t x_o)_i$\\ \hline \hline
$0$ & $(0, 0, 0)$ & $0$ & $1$ & $1$\\ \hline
& $(0, 0, 1)$ & $-2$ & $c$ & $c \pmod n$\\
$1$ & $(0, 1, 0)$ & $-1$ & $c$ & $c \pmod n$\\ 
& $(1, 0, 0)$ & $1$ & $c$ & $c \pmod n$\\ \hline
& $(0, 0, 2)$ & $-4$ & $c^2$ & $c^2 \pmod n$\\
& $(0, 1, 1)$ & $-3$ & $2 c^2$ & $2 c^2 \pmod n$\\ 
$2$ & $(0, 2, 0)$ & $-2$ & $c^2$ & $c^2 \pmod n$\\
& $(1, 0, 1)$ & $-1$ & $2 c^2$ & $2 c^2 \pmod n$\\
& $(1, 1, 0)$ & $0$ & $2 c^2$ & $2 c^2 \pmod n$\\ 
& $(2, 0, 0)$ & $2$ & $c^2$ & $c^2 \pmod n$\\ \hline
& $(0, 0, 3)$ & $-6$ & $c^3$ & $c^3 \pmod n$\\
& $(0, 1, 2)$ & $-5$ & $3 c^3$ & $3 c^3 \pmod n$\\ 
& $(0, 2, 1)$ & $-4$ & $3 c^3$ & $3 c^3 \pmod n$\\
& $(0, 3, 0)$ & $-3$ & $c^3$ & $\ast$\\
$3$ & $(1, 0, 2)$ & $-3$ & $3 c^3$ & $4 c^3 \pmod n$ {\small (\textit{including values with $\ast$})}\\ 
& $(1, 1, 1)$ & $-2$ & $6 c^3$ & $6 c^3 \pmod n$\\ 
& $(1, 2, 0)$ & $-1$ & $3 c^3$ & $3 c^3 \pmod n$\\
& $(2, 0, 1)$ & $0$ & $3 c^3$ & $3 c^3 \pmod n$\\ 
& $(2, 1, 0)$ & $1$ & $3 c^3$ & $3 c^3 \pmod n$\\
& $(3, 0, 0)$ & $3$ & $c^3$ & $c^3 \pmod n$\\ \hline
\end{tabular}
\end{table}
\begin{table}[H]
\caption{Time evolution of the LCA-UW $(T_{n, c} x)_i = c (x_{i-1} + x_{i+1} + x_{i+2}) \pmod n$ from the SSS $x_o$. 
Note that values in the table are taken modulo $n$, and empty cells represent a value of $0$.}
\label{tab:st_m01}
\centering
\vspace{1mm}
\begin{tabular}{c | c c c c c c c c c c }
\diagbox[height=1.2\line, width=1cm]{$t$}{$i$} & $-6$ & $-5$ & $-4$ & $-3$ & $-2$ & $-1$ & $0$ & $1$ & $2$ & $3$ \\ \hline
$0$ &&&&&&& $1$ &&& \\
$1$ &&&&& $c$ & $c$ & & $c$ && \\
$2$ &&& $c^2$ & $2 c^2$ & $c^2$ & $2 c^2$ & $2 c^2$ & & $c^2$ & \\
$3$ & $c^3$ & $3 c^3$ & $3 c^3$ & $4 c^3$ & $6 c^3$ & $3 c^3$ & $3 c^3$ & $3 c^3$ && $c^3$ \\
\end{tabular} 
\end{table}
\end{exm}

The following is an example of a two-dimensional LCA-UW.

\begin{exm}
\label{exm:coef02}
We consider another example where $D=2$, $J=3$, with neighborhood offsets ${\vb*{r}}_1 = (0, -1)$, ${\vb*{r}}_2 = (1, 0)$, ${\vb*{r}}_3 = (-1, -2)$.
The LCA-UW is defined by $(T_{n, c} x)_{\vb*{i}} = c (x_{\vb*{i}+\vb*{r}_1} + x_{\vb*{i}+\vb*{r}_2} + x_{\vb*{i}+\vb*{r}_3}) \pmod n$.
The condition $\sum_{j=1}^3 t_j r_j = -i$ in Equation~\eqref{eq:coef00} becomes $(t_2 -t_3, -t_1 - 2 t_3) = - \vb*{i}$, or equivalently, $(- t_2 + t_3, t_1 + 2 t_3) = \vb*{i}$.
Starting from the SSS $x_o$, the state at time $t$ is
\begin{align}
\label{eq:coef02}
(T_{n, c}^t x_o)_{\vb*{i}} &= 
\sum_{\substack{t_1 + t_2 + t_3 = t\\(- t_2 + t_3, t_1 + 2 t_3) = \vb*{i}}} \frac{t!}{{t_1}! {t_2}! {t_3}!} c^t \pmod n.
\end{align}
Calculating the values of each cell up to $t=3$ based on Equation~\eqref{eq:coef00} yields Table~\ref{tab:coef02}. 
It can be confirmed that the time evolution shown in Table~\ref{tab:st_m02} precisely follows the orbit of the LCA-UW.
\begin{table}[H]
\caption{Values of each site $\vb*{i}$ for the orbit of the LCA-UW $(T_{n, c} x)_{\vb*{i}} = c (x_{\vb*{i}+(0, -1)} + x_{\vb*{i}+(1,0)} + x_{\vb*{i}+(-1, -2)}) \pmod n$ from the SSS $x_o$ at time steps $t=0, 1, 2, 3$.}
\label{tab:coef02}
\centering
\vspace{1mm}
\begin{tabular}{c | c | c | c | l}
$t$ & $(t_1, t_2, t_3)$ & $\vb*{i}$ & $\frac{t!}{{t_1}! {t_2}! {t_3}!} c^t$ & $(T_{n, c}^t x_o)_{\vb*{i}}$\\ \hline \hline
$0$ & $(0, 0, 0)$ & $(0, 0)$ & $1$ & $1$\\ \hline
& $(0, 0, 1)$ & $(1, 2)$ & $c$ & $c \pmod n$\\
$1$ & $(0, 1, 0)$ & $(-1, 0)$ & $c$ & $c \pmod n$\\ 
& $(1, 0, 0)$ & $(0, 1)$ & $c$ & $c \pmod n$\\ \hline
& $(0, 0, 2)$ & $(2, 4)$ & $c^2$ & $c^2 \pmod n$\\
& $(0, 1, 1)$ & $(0, 2)$ & $2 c^2$ & $\ast 1$\\ 
$2$ & $(0, 2, 0)$ & $(-2, 0)$ & $c^2$ & $c^2 \pmod n$\\
& $(1, 0, 1)$ & $(1, 3)$ & $2 c^2$ & $2 c^2 \pmod n$\\
& $(1, 1, 0)$ & $(-1, 1)$ & $2 c^2$ & $2 c^2 \pmod n$\\ 
& $(2, 0, 0)$ & $(0, 2)$ & $c^2$ & $3 c^2 \pmod n$ {\small (\textit{including values with $\ast 1$})}\\  \hline
& $(0, 0, 3)$ & $(3, 6)$ & $c^3$ & $c^3 \pmod n$\\
& $(0, 1, 2)$ & $(1, 4)$ & $3 c^3$ & $\ast 2$\\ 
& $(0, 2, 1)$ & $(-1, 2)$ & $3 c^3$ & $\ast 3$\\
& $(0, 3, 0)$ & $(-3, 0)$ & $c^3$ & $c^3 \pmod n$\\
$3$ & $(1, 0, 2)$ & $(2, 5)$ & $3 c^3$ & $3 c^3 \pmod n$\\ 
& $(1, 1, 1)$ & $(0, 3)$ & $6 c^3$ & $\ast 4$\\ 
& $(1, 2, 0)$ & $(-2, 1)$ & $3 c^3$ & $3 c^3 \pmod n$\\
& $(2, 0, 1)$ & $(1, 4)$ & $3 c^3$ & $6 c^3 \pmod n$ {\small (\textit{including values with $\ast 2$})}\\ 
& $(2, 1, 0)$ & $(-1, 2)$ & $3 c^3$ & $6 c^3 \pmod n$ {\small (\textit{including values with $\ast 3$})}\\ 
& $(3, 0, 0)$ & $(0, 3)$ & $c^3$ & $7 c^3 \pmod n$ {\small (\textit{including values with $\ast 4$})}\\ \hline
\end{tabular}
\end{table}
\begin{table}[H]
\centering
\caption{Time evolution of the LCA-UW $(T_{n, c} x)_{\vb*{i}} = c (x_{\vb*{i}+(0, -1)} + x_{\vb*{i}+(1,0)} + x_{\vb*{i}+(-1, -2)}) \pmod n$ from the SSS $x_o$. 
Note that values in the table are taken modulo $n$, and empty cells represent a value of $0$.}
\label{tab:st_m02}
\begin{minipage}{0.45\textwidth}
\centering
\vspace{1mm}
\small{
$t=0$\\
\begin{tabular}{c | p{2pt} p{2pt} p{2pt} p{2pt} p{2pt} p{2pt} p{2pt}}
\diagbox[height=1.2\line, width=1cm]{$i_2$}{$i_1$} & $3$ & $2$ & $1$ & $0$ & $-1$ & $-2$ & $-3$\\ \hline
$6$ &&&&&&& \\
$5$ &&&&&&& \\
$4$ &&&&&&& \\
$3$ &&&&&&& \\
$2$ &&&&&&& \\
$1$ &&&&&&& \\
$0$ &&&& $1$ &&& \\
\end{tabular}\\
$t=2$\\
\begin{tabular}{c | p{2pt} p{2pt} p{2pt} p{2pt} p{2pt} p{2pt} p{2pt}}
\diagbox[height=1.2\line, width=1cm]{$i_2$}{$i_1$} & $3$ & $2$ & $1$ & $0$ & $-1$ & $-2$ & $-3$\\ \hline
$6$ &&&&&&& \\
$5$ &&&&&&& \\
$4$ && $c^2$ &&&&& \\
$3$ &&& $2 c^2$ &&&& \\
$2$ &&&& $3 c^2$ &&& \\
$1$ &&&&& $2 c^2$ && \\
$0$ &&&&&& $c^2$ & \\
\end{tabular}
}
\end{minipage}
\begin{minipage}{0.45\textwidth}
\centering
\vspace{1mm}
\small{
$t=1$\\
\begin{tabular}{c | p{2pt} p{2pt} p{2pt} p{2pt} p{2pt} p{2pt} p{2pt}}
\diagbox[height=1.2\line, width=1cm]{$i_2$}{$i_1$} & $3$ & $2$ & $1$ & $0$ & $-1$ & $-2$ & $-3$\\ \hline
$6$ &&&&&&& \\
$5$ &&&&&&& \\
$4$ &&&&&&& \\
$3$ &&&&&&& \\
$2$ &&& $c$ &&&& \\
$1$ &&&& $c$ &&& \\
$0$ &&&&& $c$ && \\
\end{tabular}\\
$t=3$\\
\begin{tabular}{c | p{2pt} p{2pt} p{2pt} p{2pt} p{2pt} p{2pt} p{2pt}}
\diagbox[height=1.2\line, width=1cm]{$i_2$}{$i_1$} & $3$ & $2$ & $1$ & $0$ & $-1$ & $-2$ & $-3$\\ \hline
$6$ & $c^3$ &&&&&& \\
$5$ && $3 c^3$ &&&&& \\
$4$ &&& $6 c^3$ &&&& \\
$3$ &&&& $7 c^3$ &&& \\
$2$ &&&&& $6 c^3$ && \\
$1$ &&&&&& $3 c^3$ & \\
$0$ &&&&&&& $c^3$  \\
\end{tabular}
}
\end{minipage}
\end{table}
\end{exm}

\section{Main results}
\label{sec:main}

In this section, we present the main results regarding a fundamental reduction principle. 
We demonstrate that for any coefficient $c$ of an LCA-UW satisfying $\gcd(c, n) > 1$, the generated infinite binary orbit eventually coincides with that of an LCA-UW with a reduced state size and a unit coefficient. 
Specifically, we show that the entire binary pattern of an LCA-UW can be fully characterized by examining only the $c=1$ cases across all $n \in \mathbb{Z}_{\geq 2}$.

\begin{dfn}
\label{dfn:init_conf01}
An orbit $\{T^t_{n, c} x_o\}_{t \geq 0}$ of an LCA-UW $T_{n, c}$ starting from the SSS $x_o$ \textbf{grows infinitely} if, for any $t \in {\mathbb Z}_{\geq 0}$, there exists some $\vb*{i} \in {\mathbb Z}^D$ such that $(T_{n, c}^t x_o)_{\vb*{i}} \neq 0$. 
Otherwise, the orbit \textbf{eventually vanishes} (or \textbf{does not grow infinitely}) if there exists a time step $\hat{t} \in {\mathbb Z}_{\geq 0}$ such that $(T_{n, c}^t x_o)_{\vb*{i}} = 0$ for all ${\vb*{i}} \in {\mathbb Z}^D$ whenever $t > \hat{t}$.
\end{dfn}

\begin{lem}
\label{lem:INF01}
Let the prime factorization of the number of states be $n = \prod_{j=1}^m p_j^{k_j}$, and the radical of $n$ be $\mathrm{rad}(n) = \prod_{j=1}^m p_j$. 
An orbit $\{T^t_{n, c} x_o\}_{t \geq 0}$ of an LCA-UW $T_{n, c}$ is classified into the following two cases based on the coefficient $c$. 
\begin{enumerate}
\item If $c = 0 \pmod{\mathrm{rad}(n)}$, the orbit eventually vanishes. 
\item If $c \neq 0 \pmod{\mathrm{rad}(n)}$, the orbit grows infinitely.
\end{enumerate}
\end{lem}

\begin{proof}
\begin{enumerate}
\item Suppose $c = 0 \pmod{\mathrm{rad}(n)}$. 
This implies that $c$ is a multiple of every prime factor $p_j$ of $n$. 
Let $K_n = \max_{1 \leq j \leq m} \{ k_j \}$. 
For any $t \geq K_n$, the term $c^t$ in Lemma~\ref{lem:coef00} satisfies $c^t = 0 \pmod{p_j^{k_j}}$ for all $1 \leq j \leq m$. 
By the Chinese remainder theorem, it follows that $c^t = 0 \pmod n$ for all $t \geq K_n$. 
Since every state $(T_{n, c}^t x_o)_{\vb*{i}}$ is a linear combination of terms each containing $c^t$ as a factor, the state of every cell becomes $0 \pmod n$ for all $t \geq K_n$. 
Thus, the orbit eventually vanishes.
\item Suppose $c \neq 0 \pmod{\mathrm{rad}(n)}$.
Then there exists at least one prime factor $p_j$ of $n$ such that $c$ is not divisible by $p_j$. 
By the property of prime factorization, $c^t$ is also not divisible by $p_j$ for any $t \geq 0$. 
Therefore, $c^t \neq 0 \pmod{p_j^{k_j}}$, which implies $c^t \neq 0 \pmod n$.
To show infinite growth, we identify at least one site $\vb*{i}$ where the state is non-zero for each $t$. 
Let $\vb*{r}_M$ be an extreme point (a vertex) of the convex hull of the neighborhood set $\{\vb*{r}_1, \dots, \vb*{r}_J\}$.
For the site $\vb*{i} = -t \vb*{r}_M$, the condition $\sum_{j=1}^J t_j \vb*{r}_j = -t \vb*{r}_M$ with $\sum_{j=1}^J t_j = t$ is satisfied only by the unique combination $t_M = t$ and $t_j = 0$ for all $j \neq M$. 
According to Lemma~\ref{lem:coef00}, the state at this site is given by
\begin{align}
(T_{n, c}^t x_o)_{- t \vb*{r}_M} &=  \frac{t!}{t!} c^t  = c^t \pmod n.
\end{align}
Since $c^t \neq 0 \pmod n$, the state at site $-t \vb*{r}_M$ remains non-zero for all $t \geq 0$. 
Thus, the orbit grows infinitely.
\end{enumerate}
\end{proof}

We define the operator $\mathcal{B}$ below, which applies a characteristic function that maps any positive state to $1$ and zero to $0$.

\begin{dfn}
\label{dfn:Bin}
For an orbit $\{T^t_{n, c} x_o\}_{t \geq 0}$ of an LCA-UW $T_{n, c}$ starting from $x_o$, we define the \textbf{binary projection operator} $\mathcal{B}: {({\mathbb Z} / n {\mathbb Z})}^{{\mathbb Z}^D \times {\mathbb Z}_{\geq 0}} \to {({\mathbb Z} / 2 {\mathbb Z})}^{{\mathbb Z}^D \times {\mathbb Z}_{\geq 0}}$, which maps each cell state to a binary value as follows
\begin{align}
\left(\mathcal{B} (\{T_{n, c}^t x_o \}_{t \geq 0}) \right)_{\vb*{i}, t} &=
\begin{cases} 
1 & \text{if } (T_{n, c}^t x_o)_{\vb*{i}} \neq 0,\\
0 & \text{if } (T_{n, c}^t x_o)_{\vb*{i}} = 0,
\end{cases}
\end{align}
for $\vb*{i} \in {\mathbb Z}^D$ and $t \in {\mathbb Z}_{\geq 0}$.
\end{dfn}

This operator allows us to focus on the spatio-temporal support of the orbit, disregarding the specific values in ${\mathbb Z} / n {\mathbb Z}$.
We refer to the orbit formed by the set of positive cells and zero cells as a \textbf{binary orbit}. 
Furthermore, a binary orbit that grows infinitely is called an \textbf{infinite binary orbit}.
For $K \geq 0$, when considering the infinite binary orbit $\mathcal{B}(\{T_{n, c}^t x_0\}_{t \geq 0})$ restricted to time steps $t \geq K$ for all sites $\vb*{i} \in {\mathbb Z}^D$, we denote it by $\mathcal{B} (\{T_{n, c}^t x_o \}_{t \geq K})$.

Hereafter, we shall exclusively consider infinite binary orbits.
Figure~\ref{fig:nc1} shows the classification of LCA-UW orbits $\{T^t_{n, c} x_o\}_{t \geq 0}$, where the horizontal axis represents the number of states $n$ from $2$ to $32$ and the vertical axis represents the coefficient $c$ from $1$ to $n-1$. 
In this table, the black-filled cells indicate orbits that eventually vanish, as shown in Lemma~\ref{lem:INF01}. 
We focus on the orbits represented by the white (unfilled) cells.
The values associated with these white cells will be discussed following Proposition~\ref{prop:varA}.

\begin{figure}[htb]
\centering
\includegraphics[width=1.\columnwidth]{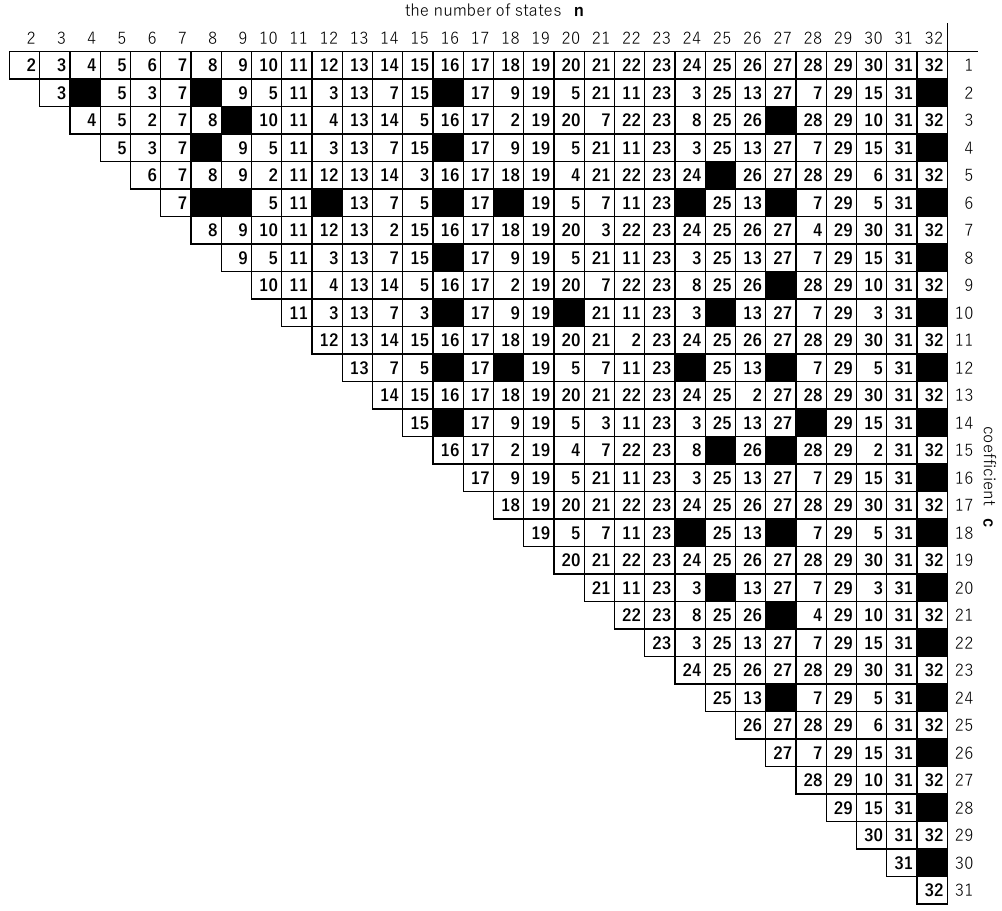}
\caption{Classification table of orbits for LCA-UW $T_{n, c}$ starting from the SSS $x_o$ ($n \leq 32$).}
\label{fig:nc1}
\end{figure}

For a fixed $n$, the orbits of $T_{n, c}$ may coincide for distinct values of $c$. 
The following lemma characterizes this equivalence.

\begin{lem}
\label{lem:gr}
Let the prime factorization of the number of states be $n = \prod_{j=1}^m p_j^{k_j}$. 
Consider a coefficient $c = \prod_{j \in I} p_j^{k_j}$ for a proper subset $I \subsetneq \{1, 2, \ldots, m\}$.
Let $\bar{I} := \{1, 2, \ldots, m\} \setminus I$ be the complement of $I$. 
Suppose $\hat{c} \in {\mathbb Z} / n {\mathbb Z}$ satisfies
\begin{align}
\hat{c} \in \left( \bigcap_{j \in I} p_j {\mathbb Z} / n {\mathbb Z} \right) \setminus \left( \bigcup_{j \in \bar{I}} p_j {\mathbb Z} / n {\mathbb Z} \right).
\end{align}
Then, for $K_c = \max_{j \in I} \{k_j\}$, the orbits of the LCA-UWs $T_{n, c}$ and $T_{n, \hat{c}}$ satisfy
\begin{align}
\mathcal{B} (\{T_{n, \hat{c}}^t x_o \}_{t \geq K_c}) = \mathcal{B} (\{T_{n, c}^t x_o \}_{t \geq K_c}).
\end{align}
\end{lem}

\begin{proof}
First, consider the case where $I = \emptyset$. 
In this case, $c \in ({\mathbb Z} / n {\mathbb Z}) \setminus \left( \bigcup_{j =1}^ m p_j {\mathbb Z} / n {\mathbb Z} \right) = ({\mathbb Z} / n {\mathbb Z})^{\times}$, which implies $K_c = 0$. 
From Lemma~\ref{lem:coef00}, for any $c \in ({\mathbb Z} / n {\mathbb Z})^{\times}$, there exists $A_{{\vb*{i}}, t} \in {\mathbb Z}_{\geq 0}$ for each time step $t$ and site $\vb*{i}$ such that
\begin{align}
(T_{n, c}^t x_o)_{\vb*{i}} = c^t A_{{\vb*{i}}, t} \pmod n.
\end{align}
Since $c \in ({\mathbb Z} / n {\mathbb Z})^{\times}$, its power $c^t$ is also invertible in ${\mathbb Z} / n {\mathbb Z}$ (i.e., $\gcd(c^t, n) = 1$), for any $t \geq 0$. 
Thus, $(T_{n, c}^t x_o)_{\vb*{i}} = 0 \pmod n$ is equivalent to $A_{\vb*{i}, t} = 0 \pmod n$. 
Similarly, $(T_{n, c}^t x_o)_{\vb*{i}}  \neq 0 \pmod n$ is equivalent to $A_{\vb*{i}, t} \neq 0 \pmod n$. 
Since $A_{\vb*{i}, t} = (T_{n, 1}^t x_o)_{\vb*{i}} \pmod n$, it follows that $\mathcal{B} (\{ T_{n, c}^t x_o \}_{t \geq 0}) = \mathcal{B} (\{ T_{n, 1}^t x_o \}_{t \geq 0})$ for all $c \in({\mathbb Z} / n {\mathbb Z})^{\times}$.

Next, we consider the case where $I \neq \emptyset$. 
From Lemma~\ref{lem:coef00}, for a given $c$, we can take $A_{{\vb*{i}}, t} \in {\mathbb Z}_{\geq 0}$ for time step $t$ and site $\vb*{i}$ such that
\begin{align}
(T_{n, c}^t x_o)_{\vb*{i}} &= c^t A_{{\vb*{i}}, t} \pmod n \quad \text{and} \quad 
(T_{n, \hat{c}}^t x_o)_{\vb*{i}} = \hat{c}^t A_{{\vb*{i}}, t} \pmod n.
\end{align}
By definition, $c = \prod_{j \in I} p_j^{k_j}$, thus $c^t = 0 \pmod{p_j^{k_j}}$ for all $j \in I$ and $t \geq 1$. 
For $\hat{c}$, since it belongs to $\bigcap_{j \in I} p_j {\mathbb Z} / n {\mathbb Z}$, it follows that $\hat{c}$ is a multiple of $\prod_{j \in I} p_j$.
Thus, for any $t \geq K_c$, the term $\hat{c}^t$ is a multiple of $\prod_{j \in I} p_j^{k_j}$, which means $\hat{c}^t = 0 \pmod{p_j^{k_j}}$ for all $j \in I$.
Furthermore, both $c$ and $\hat{c}$ are invertible modulo $p_j^{k_j}$ for all $j \in \bar{I}$ because they do not contain $p_j$ as a factor for $j \in \bar{I}$. 
Therefore, for any $t \geq K_c$, whether the state is zero in either orbit is determined solely by whether $A_{\vb*{i}, t}$ is a multiple of the remaining factors. 
Specifically, we have
\begin{align}
(T_{n, c}^t x_o)_{\vb*{i}} = 0 \pmod n & \iff A_{\vb*{i}, t} = 0 \pmod{n/c} \\
& \iff (T_{n, \hat{c}}^t x_o)_{\vb*{i}} = 0 \pmod n
\end{align}
for all $\vb*{i} \in {\mathbb Z}^D$ and $t \geq K_c$.
Therefore, $\mathcal{B} (\{T_{n, \hat{c}}^t x_o \}_{t \geq K_c}) = \mathcal{B} (\{T_{n, c}^t x_o \}_{t \geq K_c})$ holds.
\end{proof}

\begin{rmk}
Regarding Lemma~\ref{lem:gr}, it should be noted that for $t < K_c = \max_{j \in I} \{k_j\}$, the binary orbits $\mathcal{B} (\{T_{n, \hat{c}}^t x_o \}_{t < K_c})$ and $\mathcal{B} (\{T_{n, c}^t x_o \}_{t < K_c})$ may not coincide, as illustrated in Figure~\ref{fig:0001}.
Consider the time evolution of a $1$-dimensional LCA-UW $(T_{n, c} x)_i = c (x_{i-1} + x_{i} + x_{i+1}) \pmod n$ starting from $x_o$. 
Let $n = 48 = 2^4 \cdot 3$. 
If we set $c = 2^4 = 16$, we obtain $K_c = 4$. 
If we take $\hat{c} = 2$, which is a multiple of $p_1=2$ but not of $p_1^4=16$, the binary orbits are not guaranteed to match during the initial steps $t < K_c = 4$. 
Indeed, as shown in Figure~\ref{fig:0001}, the binary orbits of $(a)$ $\{T_{48, 16}^t x_o\}_{t \leq 0}$ and $(b)$ $\{T_{48, 2}^t x_o\}_{t \leq 0}$ differ for $t < 4$, but become identical for all $t \geq 4$.
\end{rmk}

\begin{cor}
\label{cor:gr_n}
For $n$, $c$, and $\hat{c}$ defined in Lemma~\ref{lem:gr}, the binary orbits of the LCA-UWs $T_{n, c}$ and $T_{n, \hat{c}}$ satisfy
\begin{align}
\mathcal{B} (\{T_{n, \hat{c}}^t x_o \}_{t \geq K_n}) = \mathcal{B} (\{T_{n, c}^t x_o \}_{t \geq K_n}),
\end{align}
where $K_n = \max_{1 \leq j \leq m} \{k_j\}$ for $n$.
Note that this equality holds because $K_n \geq K_c$ for all $c$.
\end{cor}

For a fixed state size $n$, there could be as many as $n-1$ distinct infinite binary orbits, depending on the choice of the coefficient $c$.
However, as implied by Lemma~\ref{lem:gr} and Corollary~\ref{cor:gr_n}, many of these orbits are identical. 
The following lemma shows that the number of unique orbits is significantly smaller.

\begin{lem}
\label{lem:varA01}
For a fixed number of states $n = \prod_{j=1}^m p_j^{k_j}$, as the coefficient $c$ varies from $1$ to $n-1$, the number of distinct infinite binary orbits $\mathcal{B} (\{T_{n, c}^t x_o\}_{t \geq K_n})$ is at most $2^m-1$ for $K_n = \max_{1 \leq j \leq m} \{k_j\}$.
\end{lem}

\begin{proof}
According to Lemma~\ref{lem:gr} and Corollary~\ref{cor:gr_n}, the binary orbit $\mathcal{B} (\{T_{n, c}^t x_o\}_{t \geq K_n})$ is uniquely determined by the subset of indices $I \subsetneq \{1, 2, \ldots, m\}$ such that $p_i$ divides $c$ for each $i \in I$.
Specifically, $c$ must be chosen such that the resulting orbit is infinite, which by Lemma~\ref{lem:INF01} means $c \neq 0 \pmod{\mathrm{rad}(n)}$. 
This condition implies that the index set $I$ of the prime factors dividing $c$ must be a proper subset of $\{1, 2, \ldots, m\}$.

For each subset $I$ with $|I| = l$, where $0 \leq l \leq m-1$, there are $_m \mathrm{C}_l$ possible choices.
The case $l=0$ (where $I = \emptyset$) corresponds to $c \in (\mathbb{Z} / n \mathbb{Z})^\times$.
The total number of such distinct orbits is then bounded by the number of all possible proper subsets of the $m$ prime factors
\begin{equation}
\sum_{l=0}^{m-1} {}_m \mathrm{C}_l = \left( \sum_{l=0}^{m} {}_m \mathrm{C}_l \right) - {}_m \mathrm{C}_m = 2^m - 1.
\end{equation}
This completes the proof.
\end{proof}

Let us refer again to Figure~\ref{fig:nc1}. 
For each column corresponding to a specific $n$, cells containing the same value represent LCA-UWs whose binary orbits coincide for time step $t \geq K_n$.
The following relationship holds between the number of states $n$ and the coefficient $c$, allowing us to reduce the analysis of general orbits to those with $c=1$.

\begin{lem}
\label{lem:varN01_0}
Let the number of states be $n = \prod_{j=1}^m p_j^{k_j}$ and consider a coefficient $c = \prod_{j \in I} p_j^{k_j}$ for a proper subset $I \subsetneq \{1, 2, \ldots, m\}$. 
Then, the infinite binary orbits of the LCA-UWs $T_{n, c}$ and $T_{n/c, 1}$ satisfy
\begin{align}
\mathcal{B} (\{T_{n, c}^t x_o \}_{t \geq 0}) = \mathcal{B} (\{T_{n/c, 1}^t x_o \}_{t \geq 0}).
\end{align}
\end{lem}

\begin{proof}
From Lemma~\ref{lem:coef00}, for a given $c$, there exists $A_{\vb*{i}, t} \in {\mathbb Z}_{\geq 0}$ for each site $\vb*{i}$ and time step $t$ such that
\begin{align}
(T_{n, c}^t x_o)_{\vb*{i}} &= A_{{\vb*{i}}, t} c^t \pmod n.\\
(T_{n/c, 1}^t x_o)_{\vb*{i}} &= A_{{\vb*{i}}, t} \pmod {n/c}.
\end{align}
For any $\vb*{i} \in {\mathbb Z}^D$ and $t \geq 0$, the following equivalences hold
\begin{align}
(T^t_{n, c} x_o)_{\vb*{i}} = 0 \pmod n & \iff A_{{\vb*{i}}, t} = 0 \pmod {n/c} \\
&\iff (T^t_{n/c, 1} x_o)_{\vb*{i}} = 0 \pmod {n/c}.
\end{align}
Thus, we obtain $\mathcal{B} (\{T_{n, c}^t x_o \}_{t \geq 0}) = \mathcal{B} (\{T_{n/c, 1}^t x_o \}_{t \geq 0})$.
\end{proof}

When the coefficient is fixed at $c=1$, the binary orbits corresponding to different numbers of states $n$ are distinct.

\begin{lem}
\label{lem:varN01_1} 
For $c=1$, if $n \neq \hat{n}$, then the binary orbits of the LCA-UWs $T_{n, 1}$ and $T_{\hat{n}, 1}$ satisfy
\begin{align}
\mathcal{B} (\{T_{n, 1}^t x_o \}_{t \geq 0}) \neq \mathcal{B} (\{T_{\hat{n}, 1}^t x_o \}_{t \geq 0}).
\end{align}
\end{lem}

\begin{proof}
Let $\vb*{r}_M$ be an extreme point of the convex hull of the neighborhood set $\{ \vb*{r}_1, \ldots, \vb*{r}_J \}$, and let $\vb*{r}_{\tilde{M}} \in \{ \vb*{r}_1, \ldots, \vb*{r}_J \} \setminus \{\vb*{r}_M\}$ be another neighborhood vector such that $\vb*{i} = -t \vb*{r}_M - \vb*{r}_{\tilde{M}}$ lies on the boundary of the support, ensuring that no other combination of neighborhood vectors can reach site $\vb*{i}$.
According to Lemma~\ref{lem:coef00}, for any time step $t+1$, we can consider the state at a specific site $\vb*{i}$ that is reached by choosing the vector $\vb*{r}_M$ exactly $t$ times and the vector $\vb*{r}_{\tilde{M}}$ exactly once. This site is given by $\vb*{i} = -t \vb*{r}_M - \vb*{r}_{\tilde{M}}$. 
At this site, the multinomial coefficient $A_{\vb*{i}, t+1}$, which represents the number of paths from the origin to $\vb*{i}$ in $t+1$ steps, simplifies to
\begin{align}
A_{\vb*{i}, t+1} = \frac{(t+1)!}{t! 1!} = t+1.
\end{align}
Then, the states of the orbits for $T_{n, 1}$ and $T_{\hat{n}, 1}$ at this site are
\begin{align}
(T_{n, 1}^{t+1} x_o)_{\vb*{i}} &= t+1 \pmod n, \quad (T_{\hat{n}, 1}^{t+1} x_o)_{\vb*{i}} = t+1 \pmod{\hat{n}}.
\end{align}
Without loss of generality, assume $n < \hat{n}$. By setting the time step such that $t+1 = n$, the states satisfy
\begin{align}
(T_{n, 1}^n x_o)_{\vb*{i}} &= n = 0 \pmod n, \quad (T_{\hat{n}, 1}^n x_o)_{\vb*{i}} = n \neq 0 \pmod{\hat{n}}.
\end{align}
It follows that $(\mathcal{B}\{T_{n, 1}^t x_o\}_{t \geq 0})_{\vb*{i}, n} = 0$ while $(\mathcal{B}\{T_{\hat{n}, 1}^t x_o\}_{t \geq 0})_{\vb*{i}, n} = 1$. Since there exists at least one space-time point $(\vb*{i}, n)$ where the binary values differ, we conclude
\begin{equation}
\mathcal{B} (\{T_{n, 1}^t x_o \}_{t \geq 0}) \neq \mathcal{B} (\{T_{\hat{n}, 1}^t x_o \}_{t \geq 0}).
\end{equation}
\end{proof}

Combining the preceding lemmas, we obtain the following result regarding the variations of orbits for a fixed number of states $n$ while varying the coefficient $c$.

\begin{prop}
\label{prop:varA}
For a fixed number of states $n = \prod_{j=1}^m p_j^{k_j}$ of an LCA-UW $T_{n, c}$, the infinite binary orbit $\mathcal{B} (\{T_{n, c}^t x_o\}_{t \geq K_n})$ can be classified into exactly $2^m - 1$ distinct types as the coefficient $c$ varies from $1$ to $n-1$.
\end{prop}

\begin{proof}
According to Lemma~\ref{lem:varA01}, the number of distinct infinite binary orbits is at most $2^m - 1$. 
This upper bound corresponds to the number of ways to choose a proper subset $I \subsetneq \{1, \ldots, m\}$ of prime factors that divide the coefficient $c$.
To show that these orbits are truly distinct, let $I$ and $I'$ be distinct proper subsets of indices. 
According to Lemma~\ref{lem:varN01_0}, the binary orbits of $T_{n, c}$ and $T_{n, c'}$, where $c = \prod_{j \in I} p_j^{k_j}$ and $c' = \prod_{j \in I'} p_j^{k_j}$, coincide with those of $T_{n/c, 1}$ and $T_{n/c', 1}$, respectively. 
Since $I \neq I'$, the resulting values must be distinct, i.e., $n/c \neq n/c'$. 
Finally, Lemma~\ref{lem:varN01_1} ensures that if the number of states differs, the binary orbits for the case $c=1$ are necessarily distinct. Therefore, each unique proper subset $I$ yields a unique infinite binary orbit, concluding that there are exactly $2^m - 1$ distinct types of orbits.
\end{proof}

Let us refer to Figure~\ref{fig:nc1}. 
For a fixed $n = \prod_{j=1}^m p_j^{k_j}$, we consider a coefficient $c$ satisfying the condition involving a proper subset of indices $I \subsetneq \{1, 2, \dots, m\}$ as specified in Theorem~\ref{thm:varN}.
Defining $\hat{c} = \prod_{j \in I} p_j^{k_j}$, it can be observed that the infinite binary orbit of $T_{n, c}$ eventually coincides with that of an LCA-UW with a reduced state size $n/\hat{c}$ and a unit coefficient.
This suggests that the binary projection $\mathcal{B}$ effectively acts as a quotient map that filters out the algebraic complexities introduced by $c$, reducing the system to its fundamental representation $T_{n/\hat{c}, 1}$. 
In the following, we formalize this reduction principle, which allows us to classify the entire diversity of LCA-UW patterns by focusing on the state size parameter alone.

\begin{thm}
\label{thm:varN}
Let $n = \prod_{j=1}^m p_j^{k_j}$ be the number of states with its prime factorization.
Suppose that $c \in {\mathbb Z} / n {\mathbb Z}$ satisfies
\begin{align}
c \in \left( \bigcap_{j \in I} p_j {\mathbb Z} / n {\mathbb Z} \right) \setminus \left( \bigcup_{j \in \bar{I}} p_j {\mathbb Z} / n {\mathbb Z} \right),
\end{align}
where $I \subsetneq \{1, 2, \dots, m\}$ is a proper subset of indices and $\bar{I}$ is its complement. 
Let $\hat{c} = \prod_{j \in I} p_j^{k_j}$. 
We consider the canonical surjective ring homomorphism $\phi: {\mathbb Z} / n {\mathbb Z} \to {\mathbb Z} / (n/\hat{c}) {\mathbb Z}$ defined by $\phi(a) = a \pmod{n/\hat{c}}$ for any $a \in {\mathbb Z} / n {\mathbb Z}$.
We define the map $\Phi: ({\mathbb Z} / n {\mathbb Z})^{{\mathbb Z}^D \times {\mathbb Z}_{\geq 0}} \to ({\mathbb Z} / (n/\hat{c}) {\mathbb Z})^{{\mathbb Z}^D \times {\mathbb Z}_{\geq 0}}$ induced by $\phi$ such that
\begin{align}
(\Phi \{ T_{n, c}^t x_o \}_{t \geq 0})_{\vb*{i}, t} = \phi((T_{n, c}^t {x_o}_{t \geq 0})_{\vb*{i}})
\end{align}
for any orbit $\{T_{n, c}^t x_o\}_{t \geq 0}$ with spatial index $\vb*{i} \in {\mathbb Z}^D$ and time step $t \in {\mathbb Z}_{\geq 0}$.
Then, the map $\Phi$ represents an algebraic reduction that commutes with the LCA-UW dynamics under the binary projection $\mathcal{B}$, and the following equality holds
\begin{align}
\mathcal{B} (\{T_{n, c}^t x_o \}_{t \geq K_n}) = \mathcal{B} \circ \Phi (\{T_{n, c}^t x_o \}_{t \geq K_n}) = \mathcal{B} (\{T_{n/\hat{c}, 1}^t x_o \}_{t \geq K_n}),
\end{align}
where $K_n = \max_{1 \leq j \leq m} \{k_j\}$.
Equivalently, the following diagram commutes.
\[
\xymatrix{
\{T_{n, c}^t x_o \}_{t \geq K_n} \ar[r]^{\Phi} \ar[d]_{\mathcal{B}} & \{T_{n/\hat{c}, 1}^t x_o \}_{t \geq K_n} \ar[d]^{\mathcal{B}} \\
\mathcal{B} (\{T_{n, c}^t x_o \}_{t \geq K_n}) \ar@{=}[r] & \mathcal{B} (\{T_{n/\hat{c}, 1}^t x_o \}_{t \geq K_n})
}
\]
\end{thm}

\begin{proof}
According to the classification in Proposition~\ref{prop:varA}, there are $2^m - 1$ distinct types of orbits. 
Among these, any binary orbit generated with a coefficient $c > 1$ satisfying the divisibility conditions stated above is shown to coincide with the binary orbit of $T_{n/\hat{c}, 1}$ for $t \geq K_n$, as established by Corollary~\ref{cor:gr_n}.
From Lemma~\ref{lem:varN01_0}, it holds that $\mathcal{B} (\{T_{n, \hat{c}}^t x_o \}_{t \geq 0}) = \mathcal{B} (\{T_{n/\hat{c}, 1}^t x_o \}_{t \geq 0})$.
Since $n/\hat{c}$ is an integer in $\mathbb{Z}_{\geq 2}$, the full diversity of infinite binary orbits is entirely covered by examining only the cases where the coefficient is unity across all $n \in {\mathbb Z}_{\geq 2}$.
\end{proof}

Let us refer again to Figure~\ref{fig:nc1}. 
For any given $n$, the infinite binary orbit for any $c > 1$ is uniquely determined by the orbit of $c=1$ with a corresponding smaller state size. 
Thus, all infinite binary orbits can be understood by investigating only the $c=1$ case.

\section{Concluding remarks}
\label{sec:cr}

In this study, we have investigated the diversity of infinite binary orbits generated by the LCA-UWs $T_{n, c}$. 
Our primary focus was to understand how the variation of the number of states $n$ and the coefficient $c$ influences the resulting spatio-temporal patterns from a binary perspective.
The core finding of this paper is the fundamental reduction of the parameter space. 
Through Lemma~\ref{lem:varN01_0} and Proposition~\ref{prop:varA}, we demonstrated that any orbit with a coefficient $c > 1$ eventually coincides with the orbit of an LCA-UW with a smaller state size $n/c$ and a unit coefficient $c=1$. 
This implies that the seemingly vast variety of patterns produced by different $(n, c)$ pairs is, in fact, inherently contained within the set of patterns generated by the $c=1$ cases across varying $n$.
This simplification, formalized in Theorem~\ref{thm:varN}, provides a significant advantage for the classification of LCA-UWs. 
By reducing the two-dimensional parameter search $(n, c)$ to a one-dimensional investigation of $n$, we have established a more streamlined framework for analyzing LCA-UW dynamics. 

In future work, we will use this reduction to study the specific topological properties and fractal dimensions of these orbits. 
By focusing on the $T_{n, 1}$ series, we aim to understand the full structure of LCA-UW patterns. 
We also plan to apply our classification to more general LCA starting from the SSS $x_o$. 
This will help us analyze the behavior and patterns of a much wider range of LCAs.

\section*{Acknowledgment}
This work was partly supported by Grants-in-Aid for Scientific Research (22K03435, 26K06936) funded by the Japan Society for the Promotion of Science.

\end{document}